\documentclass{article}
\usepackage[utf8]{inputenc}
\usepackage{mathtools}
\usepackage{amsfonts}
\usepackage{amsmath}
\usepackage{amsthm}
\usepackage{amssymb}

\usepackage{graphicx}
\usepackage{hyperref}
\usepackage{arxiv}

\newcommand{\oN}{{\mathbb N}}
\newcommand{\oR}{{\mathbb R}}

\newcommand\norm[1]{\left\lVert#1\right\rVert}

\newcommand{\ignore}[1]{}

    \usepackage{pgf,tikz,pgfplots}

\usepackage{mathrsfs}

\usetikzlibrary{arrows}

\title{On the Exactness of Sum-of-Squares Approximations for the  Cone of $5\times 5$  Copositive Matrices}

\author {Monique Laurent
\thanks{Centrum Wiskunde \& Informatica (CWI), Amsterdam, and Tilburg University. \url{monique.laurent@cwi.nl} }
\And 
Luis Felipe Vargas
\thanks{Centrum Wiskunde \& Informatica (CWI), Amsterdam. \url{luis.vargas@cwi.nl}
\newline
This work is supported by the European Union's Framework Programme for Research and Innovation Horizon
2020 under the Marie Skłodowska-Curie Actions Grant Agreement No. 813211  (POEMA).
} 
}

\newtheorem{theorem}{Theorem}[section]
\newtheorem{coro}[theorem]{Corollary}
\newtheorem{lemma}[theorem]{Lemma}
\newtheorem{conjecture}[theorem]{Conjecture}
\newtheorem{question}[theorem]{Question}
\newtheorem{proposition}[theorem]{Proposition}

\newtheorem{remark}[theorem]{Remark}

\newcommand{\MD}{\mathcal{D}}

\newcommand{\bCOP}{\partial\text{\rm COP}}

\usepackage{xcolor}
\makeatletter
\def\tagform@#1{\maketag@@@{$($#1$)$\@@italiccorr}}
\makeatother
\numberwithin{equation}{section}

\makeatletter
\renewcommand*\env@matrix[1][*\c@MaxMatrixCols c]{%
	\hskip -\arraycolsep
	\let\@ifnextchar\new@ifnextchar
	\array{#1}}
\makeatother

\newcommand{\supp}{\text{\rm Supp}}
\newcommand{\COP}{{\text{\rm COP}}}
\newcommand{\MK}{{\mathcal K}}
\newcommand{\LAS}{{\text{\rm LAS}}}
\newcommand{\MQ}{{\mathcal Q}}
\newcommand{\MS}{{\mathcal S}}

\newcommand{\Diag}{\text{\rm Diag}}
\newcommand{\intt}{\text{\rm int}}

\newcommand{\mmod}{\text{\rm mod }}

\begin{document}
\maketitle

\begin{abstract}
We investigate  the hierarchy of  conic inner approximations $\MK^{(r)}_n$ ($r\in \oN$) for the copositive cone $\COP_n$, introduced by  Parrilo ({\em {S}tructured {S}emidefinite {P}rograms and {S}emialgebraic {G}eometry
  {M}ethods in {R}obustness and {O}ptimization}, PhD Thesis, California Institute of Technology, 2001).
It is  known that $\COP_4=\MK^{(0)}_4$  and that, while the union of the cones $\MK^{(r)}_n$ covers the interior of $\COP_n$, it does not cover the full cone $\COP_n$ if $n\ge 6$. Here we investigate the remaining case $n=5$, where all extreme rays have been fully characterized by Hildebrand (The extreme rays of the 5 $\times$ 5 copositive cone. {\em Linear Algebra and its Applications}, 437(7):1538--1547, 2012). We show that the Horn matrix $H$ and its positive diagonal scalings play an exceptional role among the extreme rays of $\COP_5$.
We show that equality $\COP_5=\bigcup_{r\ge 0} \MK^{(r)}_5$ holds if and only if 
any positive diagonal scaling of $H$ belongs to  $\MK^{(r)}_5$ for some $r\in\oN$.
As a main ingredient for the proof, we introduce new Lasserre-type conic inner approximations for $\COP_n$, based on sums of squares of polynomials. We show their links to the cones $\MK^{(r)}_n$, and we use  an optimization approach that permits to exploit finite convergence results on Lasserre hierarchy to show membership in the new cones.

\end{abstract}

\section{Introduction}

The main object of study in this paper is the cone  of copositive matrices $\COP_n$, defined as 
\begin{align}\label{cop}
\COP_n=\{M\in \mathcal{S}^n \text{ : } x^TMx\geq 0 \text{ for all } x\in\oR^n_+ \}
\end{align}
 or,  equivalently,  as
\begin{align}\label{cop-s}
\COP_n=\{M\in \mathcal{S}^n \text{ : } (x^{\circ2})^TMx^{\circ2}\geq 0 \text{ for all } x\in\mathbb{R}^n \},
\end{align}
after setting $x^{\circ 2}=(x_1^2,\ldots,x_n^2)$ and  $\oR^n_+=\{x\in\oR^n: x\ge 0\}$.
Optimizing over the copositive cone is  hard in general since this captures a wealth of hard combinatorial optimization problems such as finding maximum stable sets in graphs and minimum  graph coloring, and,  more generally,  mixed-integer binary optimization problems (see, e.g.,  \cite{Betal}, \cite{dKP2002},  \cite{Burer2009}, \cite{DR2021}). 
Determining whether a matrix $M$ is copositive is a co-NP-complete problem (see  \cite{MK1987}). These hardness results  motivate investigating hierarchies of cones that offer tractable approximations for the copositive cone. 
Such conic approximations arise naturally by replacing  the  condition $x^TM x \ge 0$ on $\oR^n_+$ in (\ref{cop}),  or the  condition $(x^{\circ 2})^TMx^{\circ 2}\ge 0$ on $\oR^n$ in (\ref{cop-s}), by a {\em sufficient }condition for nonnegativity. 
  Parrilo  \cite{Parrilo-thesis-2000}  introduced the cones $\MK_n^{(r)}$, whose definition relies on requiring that the polynomial  $(\sum_{i=1}^nx_i^2)^r (x^{\circ 2})Mx^{\circ 2}$ is a   sum of squares of polynomials, i.e., 
\begin{align}\label{eqKr}
\MK^{(r)}_n=\Big\{M\in \MS^n: \Big(\sum_{i=1}^nx_i^2\Big)^r(x^{\circ2})^TMx^{\circ2} \in\Sigma\Big\}.
\end{align}
Here and throughout,  $\Sigma=\{\sum_{i}q_i^2: q_i\in \mathbb{R}[x]\}$ denotes  the cone of sums of squares of polynomials, $\oR[x]_r$ denotes the space of polynomials with degree at most $r$ and we set $\Sigma_r=\Sigma\cap \oR[x]_r$. 
By construction we have $\MK^{(r)}_n\subseteq \MK^{(r+1)}_n\subseteq \COP_n$ and thus 
\begin{align}\label{eqincl}
\bigcup_{r\ge 0} \MK^{(r)}_n \subseteq \COP_n.
\end{align}
In addition, the cones $\MK^{(r)}_n$ cover the interior of the copositive cone, i.e.,
\begin{align}\label{eqint}
\intt (\COP_n)\subseteq \bigcup_{r\ge 0}\MK^{(r)}_n.
\end{align}
Indeed,  a matrix $M$ lies in the interior of $\COP_n$ precisely when  $x^TMx>0$ on $\oR^n_{+}\setminus \{0\}$, or, equivalently, when $(x^{\circ2})^TMx^{\circ2} >0$ on $\oR^n\setminus \{0\}$.  Using a result of Reznick \cite{Reznick1995},  this  positivity condition implies  existence of an integer $r\in\oN$ for which $(\sum_{i=1}^nx_i^2)^r(x^{\circ2})^TMx^{\circ2}$ is a sum of squares. 

As is well-known  we have equality $\COP_n=\MK^{(0)}_n$ if and only if $n\leq 4$ \cite{Diananda}. 
For  $n\ge 6$,   copositive matrices  that do not belong to  any cone $\MK_n^{(r)}$  have  been constructed  in  \cite{LV2021b},  so  the inclusion (\ref{eqincl})  is strict for any $n\geq 6$.  
However, the question  of deciding  whether the inclusion (\ref{eqincl}) is strict for $n=5$, which is the main topic of this paper, remains open.
\begin{question}[\cite{LV2021b}]\label{q1}
Does equality $\COP_5 = \bigcup_{r\ge 0}\MK_5^{(r)}$ hold?
\end{question}
\noindent
It is known that the inclusion $\MK^{(0)}_5\subseteq \COP_5$ is strict.
For instance, the following matrix, known as the {\em Horn matrix}, 
\begin{align}\label{matHorn}
H = \left(\begin{matrix} 1 & 1 & -1 & -1 & 1\cr
1 & 1 & 1 & -1 & -1\cr
-1 & 1 & 1 & 1 & -1\cr
-1 & -1 & 1 & 1 & 1 \cr
1 & -1 & -1 & 1 & 1
 \end{matrix}\right)
\end{align}
is copositive, but it does not belong to the cone $\MK^{(0)}_5$. On the other hand,  $H$ belongs to the cone $ \MK^{(1)}_5$ (see \cite{Parrilo-thesis-2000}). Clearly, any positive diagonal scaling of a copositive matrix remains a copositive matrix, i.e., $M\in \COP_n$ implies $DMD\in \COP_n$ for any diagonal matrix $D$ with $D_{ii}>0$ for $i\in[n]$. However, this operation does {\em not } preserve the cone $\MK^{(r)}_n$ for $r\ge 1$ (see \cite{DDGH}). For instance,  $H\in \MK^{(1)}_5$, but not every positive diagonal scaling of $H$ belongs to $\MK^{(1)}_5$;  the positive diagonal scalings of  $H$ that still belong to $\MK^{(1)}_5$ are characterized in \cite{LV2021b}. It is an open question whether any positive diagonal scaling of $H$ belongs to some cone $\MK^{(r)}_5$.

\begin{question}\label{q2}
Is it true that $DHD\in \bigcup_{r\ge 0}\MK_5^{(r)}$ for all positive diagonal matrices $D$?
\end{question}
\noindent
As we will show in this paper, a positive answer to Question \ref{q2} would imply a positive answer to Question \ref{q1}.
The following is the main contribution of this paper.

\begin{theorem}\label{theomain}
Equality $\COP_5 = \bigcup_{r\ge 0}\MK_5^{(r)}$ holds if and only if $DHD\in  \bigcup_{r\ge 0}\MK_5^{(r)}$ for all positive diagonal matrices $D$.
\end{theorem}
\noindent
In view of relation (\ref{eqint}), in order to show that any $5\times 5$ copositive matrix lies in some $\MK^{(r)}_5$,   we can restrict our attention to copositive matrices that lie on the boundary $\partial \COP_5$  of the copositive cone. Moreover, it suffices to consider matrices that lie on an extreme ray of $\COP_5$.
A crucial ingredient for the proof of Theorem \ref{theomain} is the fact that all the extreme rays of the cone $\COP_5$ are known. They have been characterized by Hildebrand \cite{Hildebrand}, who proved that (up to simultaneous row/column permutation) they fall into three categories: either they are generated by a matrix in $\MK^{(0)}_5$, or they are generated by a positive diagonal scaling of the Horn matrix, or they are generated by a positive diagonal scaling of a class of special matrices  $T(\psi)$ (see Theorem \ref{theo-extreme-rays} below  for details). 
In order to show Theorem \ref{theomain} we thus need to show that all positive diagonal scalings of the matrices $T(\psi)$ lie in some cone $\MK^{(r)}_5$. This forms the main technical part of the paper, which, as we will explain below, relies on following an optimization approach.
\ignore{
\medskip 
We mention another  related open question on the copositive cone.
In \cite{DDGH}  it is shown that any copositive matrix in $\COP_5$ with an all-ones diagonal belongs to the cone $\MK^{(1)}_5$. On the other hand, in \cite{LV2021b} we constructed matrices in $\COP_n$ (for any $n\ge 7$) with an all-ones diagonal that do not belong to any of the cones $\MK^{(r)}_n$. Whether this extends or not  to the case of copositive matrices of size $n=6$ remains open.

\begin{question}[\cite{LV2021b}]
Is it true that any $6\times 6$ copositive matrix with an all-ones diagonal belongs to $\MK^{(r)}_6$ for some $r\in \oN$?
\end{question}
}



\subsubsection*{Organization of the paper}
In Section \ref{secoverview} we give an overview of the main tools and results: we introduce other  sum-of-squares hierarchies of inner approximations for $\COP_n$ (including the cones $\LAS^{(r)}_n$), we present the known description of the extreme rays of $\COP_5$ (that include the matrices $T(\psi)$ ($\psi\in \Psi$) from (\ref{T})), and we sketch the main arguments used to show the main result (Theorem \ref{theomain}). As we explain there, Theorem~\ref{theomain} reduces to showing that every positive diagonal scaling of the matrices $T(\psi)$ ($\psi\in \Psi$) belongs to some of the cones $\LAS^{(r)}_n$ (Theorem \ref{theomain2}).
In Section \ref{secrelations} we show the relationships between the various hierarchies of inner approximations for $\COP_n$. Section \ref{secproofmaintheo} is devoted to the proof of Theorem \ref{theomain2}. In the final Section \ref{secfinal} we group some conclusions, questions, and further research directions.

\subsubsection*{Notation}
Throughout we let $\MD^n$ denote the set of diagonal matrices and $\MD^n_{++}$ the set of diagonal matrices with positive diagonal entries.
For $d\in \oR^n$ we let $D=\Diag(d)\in \MD^n$ denote the diagonal matrix with diagonal entries $D_{ii}=d_i$ for $i\in [n]$. For $x\in \oR^n$, $\|x\|=\sqrt{\sum_{i=1}^nx_i^2}$ denotes its Euclidean norm and $\|x\|_1=\sum_{i=1}^n |x_i|$ denotes its $\ell_1$-norm. For a vector $x\in\oR^n$, $\supp(x)=\{i\in [n]: x_i\ne 0\}$ denotes its support. The vectors $e_1,\ldots,e_n$ denote the standard basis vectors in $\oR^n$, and we let $e=e_1+\ldots + e_n$ denote the all-ones vector.

Given polynomials $p_1,\ldots,p_k$, we let $(p_1,\ldots,p_k)=\{\sum_{i=1}^k u_ip_i: u_i\in \oR[x]\}$ denote the ideal generated by the $p_i$'s.
Throughout we denote by $I_\Delta$ the ideal generated by the polynomial $\sum_{i=1}^n x_i-1$.
For a subset $S\subseteq [n]$ and variables $x=(x_1,\ldots,x_n)$, we use the notation $x^S=\prod_{i\in S}x_i$ and, for a sequence $\beta\in \oN^n$, we set $x^\beta=x_1^{\beta_1}\cdots x_n^{\beta_n}$.

\section{Overview of  results and methods}\label{secoverview}

In this section we give a broad overview of the strategy  that we will follow to show our main result.
As indicated above, we  need to show that any positive diagonal scaling of the special matrices $T(\psi)$ (introduced below in relation (\ref{T})) lies in some cone $\MK^{(r)}_5$. In fact we will show a sharper result and show membership in another, more restricted, conic hierarchy. 
This alternative conic hierarchy arises naturally by considering other  sufficient positivity  conditions for 
the polynomials $x^TMx$ and $(x^{\circ 2})^TMx^{\circ 2}$. We begin with introducing these alternative conic approximations. 

\subsection{Alternative conic approximations for $\COP_n$}\label{sub-cones}

The definitions of  the cone $\COP_n$  in  (\ref{cop}) and in (\ref{cop-s}) rely on requiring, respectively,  nonnegativity of the polynomial $x^TMx$ on $\oR^n_+$, and nonnegativity of the polynomial $(x^{\circ 2})^TMx^{\circ 2}$ on $\oR^n$. 
Since  these two polynomials are homogeneous, this is equivalent  to requiring, respectively,  nonnegativity of $x^TMx$ on the standard simplex $\Delta_n=\{x\in \mathbb{R}^n_+:  \sum_{i=1}^{n}x_i=1\}$,
and nonnegativity of $(x^{\circ 2})^TMx^{\circ 2}$ on the unit sphere $\mathbb{S}^{n-1}=\{x\in \mathbb{R}^n: \sum_{i=1}^nx_i^2=1\}$.
In other words we can reformulate the copositive cone as
\begin{align}
\COP_n=&\{M\in \MS^n: x^TMx\geq 0 \text{ for all } x\in \Delta_n\}, \label{cop-simplex}\\
\COP_n=& \{M\in \MS^n: (x^{\circ2})^TMx^{\circ2}\geq 0 \text{ for all } x\in \mathbb{S}^{n-1}\}. \label{cop-sphere}
\end{align}
Now  we relax the nonnegativity condition and ask instead for a sufficient condition for nonnegativity  on the simplex $\Delta_n$ or on the unit sphere $\mathbb{S}^{n-1}$, in terms of sum-of-squares representations that involve the constraints defining the simplex or the sphere. This follows the commonly used approach in polynomial optimization, based on Lasserre-type relaxations (see \cite{Las2001} and overviews in, e.g.,  \cite{Las2009}, \cite{Laurent2009}), which justifies our notation below. 

For any integer $r\in \oN$, based on definition (\ref{cop-simplex}) for $\COP_n$, we define the following 
cones
\begin{align}
\LAS_{\Delta_n}^{(r)} & =\displaystyle \Big\{M\in \MS^n: x^TMx = \displaystyle \ \sigma_0+\sum_{i=1}^n\sigma_i x_i + q \ \text{ for    } \sigma_0\in \Sigma_r, \sigma_i\in \Sigma_{r-1}   \text{ and }   q\in I_\Delta        \Big\}, \label{eqLASDa} \\
\LAS_{\Delta_n, \text{P}}^{(r)} & =\Big \{M\in \MS^n: x^TMx  =\displaystyle
\sum_{S\subseteq [n],  |S|\le r}\sigma_Sx^S + q \
 \text{ for  } \sigma_S\in \Sigma_{|S|-r} \text{ and } q\in I_\Delta\Big\}.\label{eqLASDPa}
 \end{align}
Recall $I_\Delta$ is the ideal generated by $\sum_{i=1}^nx_i-1$. The index `P' used in the notation $\LAS_{\Delta_n, \text{P}}^{(r)}$ refers to the fact that the decomposition uses the preordering, which consists   of all conic combinations of products of the constraints defining the simplex with sum-of-squares polynomials as multipliers.
Clearly, for any integer $r\geq 0$, we have $$\LAS_{\Delta_n}^{(r)} \subseteq \LAS_{\Delta_n, \text{P}}^{(r)} \subseteq \COP_n.$$
Moreover, the cones $\LAS_{\Delta_n}^{(r)} $ cover the interior of the copositive cone.

\begin{lemma}\label{interior}
$\intt(\COP_n)\subseteq \bigcup_{r\geq 0}\LAS_{\Delta_n}^{(r)}.$
\end{lemma}
\noindent
This follows from Putinar's Positivstellensatz \cite{Putinar} applied to  the simplex (see Theorem \ref{theoPutinar}).

In a similar manner, 
based on definition (\ref{cop-sphere}) for $\COP_n$, we define the following  cone
\begin{align}\label{eqLASSa}
\LAS_{\mathbb{S}^{n-1}}^{(r)}=\Big\{M\in \MS^n: (x^{\circ2})^TMx^{\circ2} = \sigma_0 + u\Big(\sum_{i=1}^nx_i^2-1\Big)
 \text{ for some }  \sigma_0\in \Sigma_{r} \text{ and } u\in \mathbb{R}[x]\Big\}.
\end{align}
As we will show later (see Theorem \ref{theo-link}),
 we have the following relationships  between the various approximation cones for $\COP_n$ that were introduced above: 
  \begin{align}\label{link-cones}
 \LAS_{\Delta_n}^{(r)}\subseteq \MK_n^{(r-2)} = \LAS_{\Delta_n, \text{\rm P}}^{(r)}= \LAS_{\mathbb{S}^{n-1}}^{(2r)}\ \text{ for any } n\ge 1 \text{ and } r\ge 2.
\end{align}
A main motivation for introducing the cones $\LAS_{\Delta_n}^{(r)}$ lies in the fact that they permit to capture certain copositive matrices on the boundary of $\COP_n$, namely those matrices $M$ that arise as positive diagonal scaling of a class of matrices generating extreme rays of $\COP_5$ (see Theorem \ref{theo-extreme-rays} and Theorem \ref{theomain2} below).


\subsection{Extreme rays of $\COP_5$} 

For answering the question of whether the two cones $\COP_5$ and $\bigcup_{r\geq0}\MK_5^{(r)}$ coincide it suffices to look at the matrices that generate an extreme ray of $\COP_5$. This indeed follows directly from the fact that  any $M\in \COP_5$ can be decomposed as a finite sum of matrices generating an extreme ray. For convenience we say that a copositive matrix is {\em extreme} if it generates an extreme ray of $\COP_n$.

 A \textit{positive diagonal scaling} of a matrix $M$ is a matrix of the form $DMD$ where $D\in \MD_{++}^n$. 
 Notice that if $M$ is an extreme matrix of $\COP_n$ then every positive diagonal scaling of $M$ is also an extreme matrix. Moreover, if $M\in \COP_n$ is an extreme matrix then the same holds for every row/column permutation of $M$, i.e., for any matrix of the form  $P^TMP$, where $P$ is a permutation matrix. As observed above, positive diagonal scaling   does not preserve in general membership in $\MK^{(r)}_n$ ($r\ge 1$), however taking a row/column permutation clearly does preserve membership in any $\MK^{(r)}_n$ ($r\ge 0$).

Hildebrand \cite{Hildebrand} characterized the set of extreme matrices of $\COP_5$. For this, he defined the following matrices 
\begin{align}\label{T}
T(\psi)=
\begin{pmatrix}
1 & -\cos \psi_4 & \cos(\psi_4+\psi_5) & \cos(\psi_2+\psi_3) & -\cos\psi_3 \\
-\cos \psi_4 & 1 & -\cos\psi_5 & \cos(\psi_5+\psi_1) & \cos(\psi_3+\psi_4)  \\
\cos(\psi_4+\psi_5) & -\cos\psi_5 & 1 & -\cos\psi_1 & \cos(\psi_1+\psi_2)  \\
\cos(\psi_2+\psi_3) & \cos(\psi_5+\psi_1) & -\cos\psi_1 & 1 & -\cos\psi_2  \\
-\cos\psi_3 & \cos(\psi_3+\psi_4) & \cos(\psi_1+\psi_2)  & -\cos\psi_2 & 1  \\
\end{pmatrix},
\end{align}
where $\psi\in \oR^5$, and 
proved the following theorem.

\begin{theorem}[\cite{Hildebrand}]\label{theo-extreme-rays}
Let $M\in \COP_5$ be an extreme matrix,  and assume that $M$ is neither an element of $\MK_5^{(0)}$ nor a positive diagonal scaling (of a row/column permutation) of the Horn matrix. Then $M$ is of the form 
$$M=P\cdot D \cdot T(\psi) \cdot D \cdot P^T,$$
where $P$ is a permutation matrix, $D\in \MD_{++}^5$ and the quintuple $\psi$ is an element of the set
\begin{align}\label{psi}
\Psi =\Big\{\psi\in \mathbb{R}^5 :  \sum_{i=1}^5 \psi_i < \pi, \ \psi_i>0 \text{ for } i\in [5]\Big\}.
\end{align}
\end{theorem}
\noindent
In summary, the extreme matrices $M$ of $\COP_5$ can be divided into three categories:
\begin{description}
\item [(i)] $M\in \MK_n^{(0)}$,
\item [(ii)] $M$ is (up to row/column permutation)  a positive diagonal scaling of the Horn matrix,
\item [(iii)]   $M$ is (up to row/column permutation)  a positive diagonal scaling  of a matrix $T(\psi)$ for some $\psi \in \Psi$.
\end{description}

Our main result in this paper is to show that every matrix from the third category of extreme matrices of $\COP_5$ belongs to some cone $\LAS_{\Delta_5}^{(r)}$ and thus, in view of (\ref{link-cones}), to some cone  $ \MK_5^{(r)}$. 

\begin{theorem}\label{theomain2}
Let $D\in \MD_{++}$ be a positive diagonal matrix. Then, for all $\psi\in \Psi$, we have  $D\cdot T(\psi)\cdot D\in 
\bigcup_{r\ge 0}\LAS_{\Delta_5}^{(r)}\subseteq \bigcup_{r\ge 0} \MK^{(r)}_5$.
\end{theorem}

In view of Theorem \ref{theo-extreme-rays},    Theorem \ref{theomain} directly  follows from Theorem \ref{theomain2}.
As a direct consequence, in order to answer Question \ref{q1}, it suffices to look at the extreme matrices from the second category (i.e., at the positive diagonal scalings of the Horn matrix $H$). 

On the other hand, as we will show later in Lemma \ref{lemHornLASD},  the Horn matrix $H$ does {\em not}  belong to any of the cones $\LAS_{\Delta_5}^{(r)}$.  
Hence, in order to show that any diagonal scaling of the Horn matrix belongs to some cone $\MK^{(r)}_5$ and thus give an affirmative answer to Question~\ref{q1}, it will not be sufficient to consider the cones $\LAS_{\Delta_5}^{(r)}$.
A different, new strategy will be needed.


\subsection{Sketch of the proof}\label{secsketch}

We are left with the task of proving Theorem \ref{theomain2}. For this we follow an optimization approach and  consider the following standard quadratic program, for a given copositive matrix $M\in \COP_n$:
\begin{align}\label{sqp-M}\tag{SQP$_{M}$}
p_M^{*}:=\min\{x^TMx: x\in \Delta_n\}.
\end{align}
If $p_M^*>0$ then $M\in \text{int}(\COP_n)$ and thus,  by Lemma \ref{interior}, $M\in \bigcup_{r\geq 0}\LAS_{\Delta_n}^{(r)}$. 
Hence we may  restrict our attention to the case when $p^*_M=0$, i.e., when $M\in \partial\COP_n$.
We now consider the Lasserre sum-of-squares hierarchy for problem (\ref{sqp-M}), where, for any integer $r\ge 1$, we set
\begin{equation}\label{las-sqp}
\begin{array}{ll}
p_M^{(r)}:=\sup \{\lambda : x^TMx-\lambda  =&  \sigma_0+\sum_{i=1}^nx_i\sigma_i + q \ \text{ for some } q\in I_\Delta,\\
& \text{ and } \sigma_0\in \Sigma_r,\ \sigma_i\in \Sigma_{r-1}\}.
\end{array}
\end{equation}
Then the  bounds $p^{(r)}_M$ converge asymptotically to $p_M^*=0$. Moreover, in program (\ref{las-sqp}) the supremum is attained and thus one may replace `sup' by `max'. Hence, we have $p_M^{(r)}=0$ if and only if the matrix $M$ belongs to the cone $\LAS^{(r)}_n$. Therefore, when $p_M^*=0$, we have 
$M\in \bigcup_{r\geq 0} \LAS_{\Delta_n}^{(r)}$ if and only if the Lasserre hierarchy (\ref{las-sqp}) has finite convergence, i.e., $p_M^{(r)}=0$ for some $r$. Based on this observation, our strategy is now to show finite convergence of the Lasserre hierarchy (\ref{las-sqp}) in the case when   $M$ is a positive diagonal scaling of a matrix $T(\psi)$ with $\psi \in \Psi$. 

For this we will use a general theorem of Nie \cite{Nie} that ensures finite convergence of the Lasserre hierachy (\ref{las-sqp}) when the classical optimality conditions hold at every global minimizer (see Theorem \ref{theo-Nie} below).  In our case the global minimizers of problem (\ref{sqp-M}) are given by the zeroes of the quadratic form $x^TMx$ in $\Delta_n$, whose structure is well-understood for the matrices $M=T(\psi)$. 
See Section \ref{secproofmaintheo} for details.

\section{Relationships between sum-of-squares conic approximations for $\COP_n$}\label{secrelations}

In this section we show the relationships  from (\ref{link-cones}) between the cones $\MK^{(r)}_n$, $\LAS^{(r)}_{\Delta_n}$, $\LAS^{(r)}_{\Delta_n,P}$ and $\LAS^{(r)}_{\mathbb S^{n-1}}$ introduced    in the previous sections. In addition, we  highlight the relationship to the cones  $\MQ_n^{(r)}$ introduced in \cite{PVZ2007} and point out  how   these cones can all be seen as distinct variations within a common framework.

\subsection{Links between the cones  $\MK^{(r)}_n$, $\LAS^{(r)}_{\Delta_n}$, $\LAS^{(r)}_{\Delta_n,P}$, and $\LAS^{(r)}_{\mathbb S^{n-1}}$ } 

Here we show the following result, which  establishes the links announced in relation (\ref{link-cones}) between the various cones  defined in previous sections. This result is implicitly shown in \cite{LV2021a} (see {Corollary 3.9}), where we compared different bounds for standard quadratic programs obtained via sums of squares of polynomials.

 \begin{theorem}\label{theo-link}
 Let $r\geq 2$ and $n\geq 1$, then we have
 $$\LAS_{\Delta_n}^{(r)}\subseteq\MK_n^{(r-2)}=\LAS_{\Delta_n,P}^{(r)}=\LAS_{\mathbb{S}^{n-1}}^{(2r)}.$$ 
 \end{theorem}

We begin with observing that in the definition (\ref{eqLASDPa}) of the cone $\LAS^{(r)}_{\Delta_n,P}$ we may assume that the summation only involves sets $S\subseteq [n]$ with $|S|\equiv r$ $(\mmod 2)$.

\begin{lemma}\label{lemLASDPb}
We have 
$$
\LAS^{(r)}_{\Delta_n,P}=\Big\{M\in\MS^n: x^TMx =\sum_{\substack{S\subseteq [n], |S|\le r \\ |S|\equiv r (\mmod 2)}}  \sigma_S x^S+ q \ \text{ with } \sigma_S\in \Sigma_{r-|S|} \text{ and } q\in I_\Delta\Big\}.
$$
\end{lemma}

\begin{proof}
To see this consider a term $x^S \sigma_S$, where $|S|\le r$,  $|S|\not\equiv r$ $(\mmod 2)$ and $\sigma_S\in \Sigma_{r-|S|}$.
Then $|S|\le r-1$, $\deg(\sigma_S)\le r-|S|-1$ and thus, modulo the ideal $I_\Delta$, we can replace $x^S\sigma_S$ by $x^S\sigma_S (\sum_{i=1}^n x_i)$. Now expand this expression as $\sum_{i\in S} x^{S\setminus \{i\}} \cdot \sigma_S x_i^2 
+ \sum_{i\in [n]\setminus S} x^{S\cup\{i\}} \sigma_S$. So each term in this summation is of the form $x^{T}\sigma_{T}$ with $|T| \le r$, $|T|\equiv r$ ($\mmod 2$),  and $\deg(\sigma_T)\le r-|T|$.
\end{proof}

Next we recall an alternative definition of the cone $\MK^{(r)}_n$,  following from a result in  \cite{ZVP2006}.

\begin{theorem}[\cite{ZVP2006}, Prop 9]\label{theoZVP} 
Let $f\in \oR[x]$ be a homogeneous polynomial with $\deg(f)=d$. Then the polynomial $f(x^{\circ2})$ is a sum of squares if and only if  $f$ admits a decomposition of the  form 
  \begin{align}\label{eqf}
  f = \sum_{\substack{S\subseteq [n], |S|\leq d \\ |S|\equiv d\ (\mmod  2)} }\sigma_Sx^S \ \ \text{ for some } \sigma_S\in \Sigma_{d-|S|}.
  \end{align}
In particular, for any $r\ge 0$, we have
\begin{align}\label{eqKrb}
\MK^{(r)}_n=\Big\{M\in \MS^n: \Big(\sum_{i=1}^rx_i\Big)^{r} x^TMx =\sum_{\substack{S\subseteq [n], |S|\leq r+2 \\ |S|\equiv r\ (\mmod  2)} }\sigma_Sx^S \ \ \text{ for some } \sigma_S\in \Sigma_{r+2-|S|}\Big\}.
\end{align}
\end{theorem}

Note the similarity between the description of $\LAS^{(r)}_{\Delta_n,P}$ in Lemma \ref{lemLASDPb} and that of $\MK^{(r-2)}_n$ in relation (\ref{eqKrb}). The difference lies in the fact that for $\LAS^{(r)}_{\Delta_n,P}$ we have a representation of $x^TMx$ modulo the ideal $I_\Delta$, while for $\MK^{(r-2)}_n$ we have a representation of $(\sum_i x_i)^{r-2}x^TMx$. The next lemma (whose main idea was already used, e.g., in \cite{dKLP}) gives a simple trick, useful to navigate between these two types of representations.
 
\begin{lemma}\label{lemequiv}
Let $f,g\in \oR[x]$ and assume  $f$ is homogeneous.
  The following assertions hold.  
\begin{itemize}
\item[(i)] If $(\sum_{i=1}^n x_i)^r f(x)=g(x)$, then $f-g\in I_\Delta$.
\item[(ii)] Let $\deg(f)=d$, $\deg(g)=d+r$ ($r\in \oN$), and define $\widetilde g(x)= (\sum_{i=1}^n x_i)^{d+r}g(x/(\sum_{i=1}^nx_i) )$. 
Then, $\widetilde g$ is a homogeneous polynomial of degree $d+r$. Moreover, if  $f-g\in I_\Delta$, then 
 $(\sum_{i=1}^n x_i)^r f(x)=\widetilde g(x)$. 
 \end{itemize}
 \end{lemma}

\begin{proof}
The assertion  (i) follows by expanding $(\sum_{i=1}^n x_i)^r$ 
as the sum 
$$\Big(\sum_{i=1}^n x_i\Big)^r=\Big(\sum_{I=1}^n x_i -1 +1\Big)^r= 1 +\Big(\sum_{i=1}^n x_i-1\Big)\Big(\sum_{k=1}^r {r\choose k} \Big(\sum_{i=1}^n x_i-1\Big)^{k-1}\Big).$$ We now show  (ii). The claim that $\widetilde g$ is a homogeneous polynomial of degree $d+r$ is easy to check. Assume now $f-g\in I_\Delta$. 
By evaluating $f-g$ at  $x/(\sum_{i=1}^nx_i)$, we obtain $f(x/(\sum_{i=1}^nx_i) )=g(x/(\sum_{i=1}^nx_i ))$. As $f$ is homogeneous of degree $d$ this implies   $f(x)=(\sum_{i=1}^d x_i)^d g(x/(\sum_{i=1}^nx_i ))$, and the result follows after multiplying both sides by   $(\sum_{i=1}^n x_i)^r$.
\end{proof}

We will also use the following simple fact.

\begin{lemma}\label{lemsquare}
Let $\sigma\in \Sigma_k$ and 
define $\widetilde \sigma(x)= (\sum_{i=1}^n x_i)^k \sigma(x/(\sum_{i=1}^n x_i))$. Then $\widetilde \sigma$ is a homogeneous polynomial of degree $k$. Moreover,
\begin{itemize} \item[(i)] If $k\equiv \deg(\sigma)$ ($\mmod 2$), then $\widetilde \sigma\in \Sigma$.
\item[(ii)]
If $k\not\equiv \deg(\sigma)$ ($\mmod 2$), then $\widetilde \sigma = (\sum_{i=1}^n x_i)\widehat \sigma$, where $\widehat \sigma\in \Sigma$.
\end{itemize}
\end{lemma}

\begin{proof}
Note that $ \widetilde \sigma = (\sum_{i=1}^n x_i)^{k-\deg(\sigma)} \sigma'$, where 
$\sigma':= (\sum_{i=1}^nx_i)^{\deg(\sigma)}  \sigma(x/(\sum_{i=1}^n x_i))$ is a homogeneous polynomial with degree $\deg(\sigma)$. It suffices now to observe that     $(\sum_{i=1}^n x_i)^{k-\deg(\sigma)}$ is a square if $k-\deg(\sigma)$ is even, and it is a square times $(\sum_i x_i)$ if $k-\deg(\sigma)$ is odd. \end{proof}

Using these two lemmas we can now relate the two cones $\LAS^{(r)}_{\Delta_n,P}$ and $\MK^{(r-2)}_n$.

\begin{lemma}\label{lemLASDPKr}
For any $r\ge 2$, we have  $\LAS^{(r)}_{\Delta_n,P}=\MK^{(r-2)}_n.$
\end{lemma}

 \begin{proof}
First assume $M\in \LAS^{(r)}_{\Delta_n,P}$. Using Lemma \ref{lemLASDPb}, we have a decomposition  of the form 
 $x^TMx=g(x)+ q(x)$, where $q\in I_\Delta$ and $g(x)=\sum_{|S|\le r, |S|\equiv r (\mmod 2)} \sigma_S x^S$, with $\sigma_S\in\Sigma_{r-|S|}$.
Using  Lemma~\ref{lemequiv}(ii) we get
$$ \Big(\sum_{i=1}^n x_i\Big)^{r-2}x^TMx =\Big(\sum_{i=1}^nx_i\Big)^rg\Big({x\over \sum_ix_i}\Big)= \sum_{|S|\le r, |S|\equiv r (\mmod 2)} 
x^S\underbrace{\Big(\sum_{i=1}^nx_i\Big)^{r-|S|}\sigma_S\Big({x\over \sum_i x_i}\Big)}_{=\widetilde\sigma_S(x)}.$$
As $r-|S|\equiv \deg(\widetilde\sigma_S)  (\mmod 2)$ we have $\widetilde \sigma_S\in \Sigma_{r-|S|}$ by Lemma \ref{lemsquare}(i).
In view of relation (\ref{eqKrb}), this  shows that $M\in \MK^{(r-2)}_n$.
 
 Conversely, assume $M\in\MK^{(r-2)}_n$. Then, in view of (\ref{eqKrb}), we have a decomposition of the form
 $(\sum_{i=1}^n)^{r-2}x^TMx =\sum_{|S|\le r, |S|\equiv r (\mmod 2)} \sigma_S x^S$, where $\sigma_S \in\Sigma_{r-|S|}$.
 By applying Lemma~\ref{lemequiv}(i), we obtain 
 $x^TMx=\sum_{|S|\le r, |S|\equiv r (\mmod 2)} \sigma_S x^S+q$, where $q\in I_\Delta$. Combining with Lemma \ref{lemLASDPb} this shows $M\in \LAS^{(r)}_{\Delta_n,P}$.
  \end{proof}
 
To complete the proof of Theorem \ref{theo-link} we now establish the relation to  the cone $\LAS^{(r)}_{\mathbb S^{n-1}}$, which follows from a result in \cite{dKLP}.

\begin{proposition}[\cite{dKLP}]
Let $f$ be a homogeneous polynomial of  degree $2d$ and $r\in \oN$. Then, $f\Big(\sum_{i=1}^nx_i^2\Big)^r \in \Sigma $ if and only if $f=\sigma+u(\sum_{i=1}^nx_i^2-1)$ for some $\sigma\in \Sigma_{2r+2d}$ and  $u\in \oR[x]$.

In particular, for any $r\ge 2$ we have 
\begin{align}
 \LAS^{(r)}_{\mathbb S^{n-1}} =\Big\{M\in \MS^n:  \Big(\sum_{i=1}^n x_i^2\Big)^{r-2} (x^{\circ 2})^TMx^{\circ 2} \in \Sigma\Big\}=\MK^{(r-2)}_n.\label{eqLASSb}
 \end{align}
  \end{proposition}
  




We conclude this section with a reformulation for the cone $\LAS^{(r)}_{\Delta_n}$ in the same vein as the reformulation of  $\LAS^{(r)}_{\Delta_n,P}$ in Lemma \ref{lemLASDPb}.

\begin{lemma}
Let $r\ge 2$. 
If  $r$ is odd, then  we have
\begin{align}
\LAS^{(r)}_{\Delta_n}=\Big\{M\in \MS^n: \Big(\sum_{i=1}^n x_i\Big)^{r-2} x^TMx =\sum_{i=1}^n \sigma_i x_i\text{ with }  \sigma_i\in \Sigma_{r-1}\Big\}.\label{eqLASDbodd}
\end{align}
If $r$ is even and $r\ge 4$, then we have $\LAS^{(r)}_{\Delta_n}=\LAS^{(r-1)}_{\Delta_n}$.
\end{lemma}

\begin{proof}
The proof is  similar to that of Lemma \ref{lemLASDPb}, except  we now have a summation that involves only sets $S\subseteq [n]$ with $|S|\le 1$. We spel out the details for clarity. Consider first the case when  $r$ is odd.
Assume $M\in \LAS^{(r)}_{\Delta_n}$, so that $x^TMx =\sigma_0+\sum_{i=1}^n \sigma_ix_i+q$, where $q\in I_\Delta, \sigma_0\in \Sigma_r, \sigma_i\in \Sigma_{r-1}$. 
Combining  Lemma \ref{lemequiv}(ii) and Lemma~\ref{lemsquare}  we obtain a decomposition as in  (\ref{eqLASDbodd}). 
Conversely, starting from a decomposition as in (\ref{eqLASDbodd}) we get a decomposition as in (\ref{eqLASDa}) by applying Lemma \ref{lemequiv}(i).

Consider now the case $r\ge 4$ even. 
Assume $M\in  \LAS^{(r)}_{\Delta_n}$, we show $M\in  \LAS^{(r-1)}_{\Delta_n}$. Starting from a decomposition as in (\ref{eqLASDa}) and using as above Lemma \ref{lemequiv}(i) and Lemma~\ref{lemsquare},  we obtain a decomposition 
$$(\sum_{j=1}^nx_j)^{r-2}x^TMx =\widetilde\sigma_0+(\sum_{j=1}^nx_j)\sum_{i=1}^n \widetilde \sigma_ix_i,$$ where $\widetilde \sigma_0\in \Sigma_r$ and $\widetilde \sigma_i\in \Delta_{r-1}$. 
From this it follows that  the polynomial $\sum_{j=1}^n x_j$ divides $\widetilde \sigma_0$, which implies its square divides $\widetilde\sigma_0$. Then we can divide out by $\sum_{j=1}^nx_j$ and obtain an expression as in (\ref{eqLASDbodd}) (replacing $r$ by $r-1$), that certifies membership of $M$ in $\LAS^{(r-1)}_{\Delta_n}$.
\end{proof}



\subsection{Link to the cones $\MQ^{(r)}_n$}
 
The definition (\ref{eqKrb}) of the cone $\MK^{(r)}_n$ 
 involves only square-free monomials, of the form $x^S=\prod_{i\in S}x_i$. As observed in \cite{ZVP2006,PVZ2007}, one can allow arbitrary monomials and, after using again the argument of Lemma \ref{lemequiv}, we get the following  alternative definitions
 \begin{align}
\MK^{(r)}_n& =\Big\{M\in\MS^n: \Big(\sum_{i=1}^n x_i\Big)^r x^TMx =
\sum_{\substack{\beta\in \mathbb{N}^n\\ |\beta|\leq r+2}} \sigma_\beta x^\beta \ \text{ for  } \sigma_\beta\in \Sigma_{r+2-|\beta|}\Big\}\label{eqKrc}\\
&= \Big\{M\in\MS^n:  x^TMx =\sum_{\substack{\beta\in \mathbb{N}^n\\ |\beta|\leq r+2}} \sigma_\beta x^\beta + q\ \text{ for  } \sigma_\beta\in \Sigma_{r+2-|\beta|} \text{ and } q\in I_\Delta\Big\}.\label{eqKrd}
\end{align}
Based on relation (\ref{eqKrc}), the authors of \cite{PVZ2007} proposed the cones $\MQ_n^{(r)}$, that are defined as the variation (\ref{eqQra}) of (\ref{eqKrc}) obtained by just considering the terms associated to the  monomials $x^\beta$ with highest  degree $r$ or $r+2$. In other words,
\begin{align}
\MQ^{(r)}_n& =\Big\{M\in \MS^n: \Big(\sum_{i=1}^n x_i\Big)^r x^TMx =
\sum_{\substack{\beta\in \mathbb{N}^n\\ |\beta|=r, r+2 }} \sigma_\beta x^\beta
\ \text{ for  } \sigma_\beta\in \Sigma_{r+2-|\beta|}\Big\}, \label{eqQra}\\
&=\Big\{M\in \MS^n: x^TMx =
\sum_{\substack{\beta\in \mathbb{N}^n\\ |\beta|=r, r+2 }} \sigma_\beta x^\beta + q \ \text{ for  } \sigma_\beta\in \Sigma_{r+2-|\beta|} \text{ and } q\in I_\Delta\Big\}, \label{eqQrb}\end{align}
where  the equivalence of (\ref{eqQra}) and (\ref{eqQrb}) follows again using Lemma \ref{lemequiv}.
Clearly, we have  inclusion $\MQ^{(r)}_n\subseteq \MK^{(r)}_n$ for all $n\ge 1$ and $r\ge 0$, with equality $\MQ^{(r)}_n=\MK^{(r)}_n$ for $r=0,1$.

\begin{remark}
For any $r\ge 2$,  the two cones  $\LAS^{(r)}_{\Delta_n}$ and $\MQ^{(r-2)}_n$ are both contained in  the cone $\MK^{(r-2)}_n$. 
 In view of (\ref{eqKrd}), membership in $\MK^{(r-2)}_n$ requires a decomposition using terms  of the form $x^\beta\sigma_\beta$
for all $\beta$ such that $|\beta|\le r$ and  $|\beta|\equiv r (\mmod 2)$. In view of (\ref{eqLASDa}), for membership in $\LAS^{(r)}_{\Delta_n}$,  we consider only the terms $x^\beta\sigma_\beta$ with lowest degree $|\beta|=0,1$. On the other hand, in view of (\ref{eqQrb}),  for membership in $\MQ^{(r-2)}_n$, we consider only the terms with highest degree $|\beta|=r,r-2$. Hence, it is interesting to note that  the two cones $\LAS^{(r)}_{\Delta_n}$ and $\MQ^{(r-2)}_n$ use the ``two opposite ends'' of the spectrum of possible degrees for the terms $x^\beta\sigma_\beta$. 
\end{remark}

We conclude this section with observing that, while  the Horn matrix $H$ belongs to $\MK^{(1)}_5$, it in fact does not belong to any of the cones $\LAS_{\Delta_n}^{(r)}$. The proof exploits the fact that the quadratic form $x^THx$ has infinitely many zeros in the simplex $\Delta_n$.

\begin{lemma}\label{lemHornLASD}
For all $r\in \oN$ we have $H\notin \LAS_{\Delta_5}^{(r)}$.
\end{lemma}

\begin{proof}
Assume by contradiction that $H\in \LAS_{\Delta_5}^{(r)}$ for some $r\in \oN$, i.e., 
$$x^THx= \sigma_0 + \sum_{i=1}^nx_i\sigma_i +q(x)(\sum_{i=1}^nx_i-1),$$
for some $\sigma_0, \sigma_i\in \Sigma$ and $q(x)\in \mathbb{R}[x]$. For a fixed scalar $t\in (0,1)$, consider the vector  $u_t=(\frac{1}{2},0, \frac{t}{2},\frac{1-t}{2}, 0)\in \Delta_5$, which can be verified to define a zero of $x^THx$, i.e., $u_t^THu_t=0$. By evaluating  the quadratic form $x^THx$ at the point $x+u_t$ we obtain 
\begin{align*}
 (x+u_t)^TH(x+u_t)
= & \sigma_0(x+u_t) + \sum_{i=1}^5(x+u_t)_i\sigma_i(x+u_t) +q(x+u_t)(\sum_{i=1}^nx_i). \label{eqH1}
\end{align*} 
As $u_t^THu_t=0$ and $x^THu_t= x_2t+(1-t)x_5$  we obtain
\begin{equation} \label{eqH}
\begin{array}{ll}
x^THx + 2x_2t+2(1-t)x_5  
 =  & \sigma_0(x+u_t) + \frac{1}{2}\sigma_1(x+u_t) + \frac{t}{2}\sigma_3(x+u_t) + \frac{1-t}{2}\sigma_4(x+u_t)\\
 & + \sum_{i=1} ^5 x_i\sigma_i(x+u_t) +q(x+u_t)(\sum_{i=1}^nx_i).
\end{array} \end{equation} 
We now compare some coefficients of the monomials (in $x$) in both sides of (\ref{eqH}) in order to reach a contradiction. As there is no constant term in the left hand side,  the constant term in the right hand side is equal to 0. This gives $\sigma_0(u_t)+\sigma_1(u_t)/2 +t\sigma_3(u_t)/2 +(1-t)\sigma_4(u_t)=0$ and thus 
$\sigma_i(u_t)=0$ for $i=0,1,3,4$.  As $\sigma_i(x+u_t)$ is a sum-of-squares polynomial in $x$ this in turn  implies that there is no linear term in $x$ in each of the polynomials $\sigma_i(x+u_t)$ for $i=0,1,3,4$. 
Next, combining this with the fact that the coefficient of $x_1$ in the left hand side is equal to 0, one obtains  that the polynomial $q(x+u_t)$ has no constant term (i.e., $q(u_t)=0$). Now we compare the coefficients of $x_2$ in both sides. In the left hand side it is equal to $2t$, while in the right hand side it is equal to $\sigma_2(u_t)$. Hence we have $2t=\sigma_2(u_t)$. We now reach a contradiction since $\sigma_2(u_t)$ is a sum-of-squares polynomial in $t$.
\end{proof}

We now show that the cones $\LAS^{(r)}_n$ cover the copositive cone only in the case $n= 2$.

\begin{proposition}
We have $\COP_2=\LAS_{\Delta_2}^{(3)}$, and the inclusion $\bigcup_{r\ge 0}\LAS_{\Delta_n}^{(r)}\subset \COP_n$ is strict for any $n\ge 3$.
\end{proposition}

\begin{proof}
First, assume $M=\left(\begin{matrix}a & c \cr c & b\end{matrix}\right)\in \COP_2$, we show $M\in \LAS^{(3)}_{\Delta_2}$. Note that $a,b\ge 0$ and $c\ge -\sqrt{ab}$ (using the fact that $u^TMu\ge 0$ with $u=(1,0),$ $ (0,1),$ and $(\sqrt b,\sqrt a)$). Then we can write
$x^TMx= (\sqrt ax_1 -\sqrt b x_2)^2 + 2(c+\sqrt{ab})x_1x_2$, which, modulo the ideal $I_\Delta$,  is equal to 
$ (\sqrt ax_1 -\sqrt b x_2)^2 (x_1+x_2)+ 2(c+\sqrt{ab})(x_2^2x_1+x_1^2x_2)$, thus showing   
$M\in \LAS^{(3)}_{\Delta_2}$.

Assume now $n=3$, we show that the  matrix $$M= \begin{pmatrix} 0 & 1 & 0 \\ 1 & 0 & 0 \\ 0 & 0 & 0 \end{pmatrix}\in \COP_3$$ does not belong to any of the cones $\LAS^{(r)}_3$.
The proof follows a similar argument as the one used for Lemma \ref{lemHornLASD}, using the fact that $u_t=(t,0,1-t)$ defines a zero of $M$ for any $t\in (0,1)$, i.e., $u_t^TMu_t=0$.
\end{proof}

\section{Proof of Theorem \ref{theomain2}}\label{secproofmaintheo}

We now proceed to prove Theorem \ref{theomain2}. 
As mentioned in Section \ref{secsketch}, we will follow an optimization approach, which allows us to apply a result of Nie \cite{Nie} as a key ingredient for our proof. We proceed in three steps.
First, we recall the sum-of-squares Lasserre hierarchy for a general polynomial optimization problem and the result of Nie \cite{Nie}, that shows   finite convergence of this  hierarchy under the classical optimality conditions. Second, applying this result to a class of standard quadratic programs, we obtain a set of sufficient conditions for a matrix $M\in \bCOP_n$, that permit to claim that any positive diagonal scaling of $M$ belongs to some cone $\LAS^{(r)}_{\Delta_n}$.
Finally, we show that these sufficient conditions hold for the matrices $T(\psi)$ ($\psi\in \Psi$), which concludes the proof of Theorem \ref{theomain2}.
 
\subsection{Optimality conditions and finite convergence of  Lasserre hierarchy}\label{opt-con}

 In this section we recall a useful general result of Nie \cite{Nie} that gives sufficient conditions for having finite convergence of the Lasserre hierarchy for a general polynomial optimization problem. 

 Given  $n$-variate polynomials $f$, $g_j$ for $j\in[m]$, and $h_i$ for $i\in[k]$,  consider the general polynomial optimization problem
 \begin{align}\label{poly-opt}\tag{Poly-Opt}
 f_{\min} = \inf_{x\in K}  f(x),
 \end{align}
 where $K$ is the semialgebraic set defined by 
 $$K=\{x\in \mathbb{R}^n: g_j(x)\geq 0 \text{ for } j\in[m],\ h_i(x)=0 \text{ for } i\in[k]\}.$$
 We say that  the \textit{Archimedean condition} holds if there exists $N\in \mathbb{R}$ such that 
 \begin{align}\label{eqArchimedean}
 N-\sum_{i=1}^nx_i^2=\sigma_0 + \sum_{j=1}^m \sigma_jg_j + \sum_{i=1}^ku_ih_i \ \text{ for some } \sigma_0, \sigma_j\in \Sigma \text{ and } u_i\in \mathbb{R}[x].
 \end{align}
For any integer $r\in \oN$ consider the corresponding Lasserre sum-of-squares hierarchy 
 \begin{align}\label{las-poly}
 \begin{split}
f^{(r)}=\sup\Big\{\lambda : f -\lambda = \sigma_0+\sum_{j=1}^m\sigma_jg_j + \sum_{i=1}^{k}& u_ih_i(x)\text{ for some }  \sigma_0\in \Sigma_r, \sigma_i\in \Sigma_{r-\text{deg}(g_i)} \\ & \text{ and } u_i\in \mathbb{R}[x]_{r-\text{deg}(h_i)} \Big\}.
\end{split}
 \end{align}
  By the following result of Putinar \cite{Putinar}, under the Archimedean condition, asymptotic convergence is guaranteed, i.e., $f^{(r)}\to f_{\min}$ as $r\to \infty$.

\begin{theorem}[\cite{Putinar}]\label{theoPutinar}
Assume  $\MK$ satisfies the Archimedean condition (\ref{eqArchimedean}). If  a polynomial $p$ is strictly positive on $K$,  then $p$ can be written as  $p= \sigma_0 + \sum_{j=1}^m \sigma_jg_j + \sum_{i=1}^ku_ih_i$ for some $\sigma_0, \sigma_j\in \Sigma$ ($j\in[m]$) and $u_i\in \mathbb{R}[x]$ ($i\in [k]$). 
\end{theorem}

The Lasserre hierarchy is said to have  {\em finite convergence} if $f^{(r)}=f_{\min} $ for some $r\in \mathbb{N}$.  In general, finite convergence is not always achieved. However,  Nie  showed  in \cite{Nie} a very useful result that permits to show finite convergence of the Lasserre hierarchy under some extra conditions apart from the Archimedean condition. These conditions rely on the classical optimality conditions, that we now recall (see, e.g., the textbook \cite{Bertsekas1999}).

\medskip
Let $u$ be a local minimizer of problem (\ref{poly-opt}) and let $J(u)=\{j\in [m]: g_j(u)=0\}$ be the set of inequality constraints that are active at $u$. Then the  \textit{constraint qualification constraint} (abbreviated as CQC) holds at $u$ if the set $\{\nabla g_j(u) : j\in J(u)\} \cup \{\nabla h_i(u): i\in [k]\}$ is linearly independent.  if CQC holds at $u$ then there exist scalars $\lambda_1, \dots, \lambda_k, \mu_1, \dots, \mu_m\in\oR$  satisfying
\begin{align*}
  \nabla f(u)=\sum_{i=1}^{k}\lambda_i\nabla h_i(u)+\sum_{j\in J(u)}\mu_j\nabla g_j(u), \quad \mu_j\ge 0 \text{ for } j\in J(u),\quad \mu_j=0\text{ for } j\in [m]\setminus J(u).
\end{align*}
If, in addition,  $\mu_j>0$ holds for all $j\in J(u)$, then  one says that the \textit{strict complementarity condition} (abbreviated as SCC) holds. Let $L(x)$ the Lagrangian function, 
defined by
$$L(x)= f(x)-\sum_{i=1}^{k}\lambda_ih_i(x)-\sum_{j\in J(u)}\mu_jg_j(x).$$
Another necessary condition for $u$ to be a local minimizer is the following inequality 
\begin{align}\label{eqSONC}\tag{SONC}
    v^T\nabla^2L(u)v\geq 0 \text{ for all } v\in G(u)^{\perp},
\end{align}
where $G(u)$ is defined by
$$G(u)^{\perp}=\{x\in \oR^n: x^T\nabla g_j(u)=0 \text{ for all  } j\in J(u) \text{ and }  x^T\nabla h_i(u)=0 \text{ for all } i\in [k]\}.$$
 If it happens that the inequality (\ref{eqSONC}) is strict, i.e., if
\begin{align}\label{eqSOSC}\tag{SOSC}
    v^T\nabla^2L(u)v> 0 \text{ for all } 0\neq v\in G(u)^{\perp},
\end{align}
then one says that the \textit{second order sufficiency condition}  (SOSC) holds at $u$.

\medskip
We can now state the following result by Nie \cite{Nie}.

\begin{theorem}[\cite{Nie}]\label{theo-Nie}
Assume the Archimedean condition (\ref{eqArchimedean}) holds for the polynomial tuples $h$ and $g$ in problem (\ref{poly-opt}). If the constraint qualification condition (CQC), the strict complementarity condition (SCC), and the second order sufficiency condition (SOSC)  hold at every global minimizer of (\ref{poly-opt}), then the Lasserre hierarchy (\ref{las-poly}) has finite convergence, i.e., $f^{(r)}=f_{\min}$ for some $r\in \oN$.
\end{theorem}

In the next section we will apply Theorem \ref{theo-Nie} to a class of standard quadratic programs in order to show finite convergence of the corresponding Lasserre hierarchy.
One important observation, already made in \cite{Nie}, is that this   strategy 
can only work when the number of global minimizers is finite. 

\subsection{Optimality conditions for standard quadratic programs}

Consider a matrix $M\in \bCOP_n$. The objective of this section is to give sufficient conditions on $M$ that permit to conclude that  $DMD\in \bigcup_{r\geq0}\LAS_{\Delta_n}^{(r)}$ for all $D\in \MD_{++}^{n}$. This will be very useful since,  in the next section, we will show that the matrices $T(\psi)$ ($\psi\in \Psi$) satisfy these sufficient conditions and thus we will be able to conclude the proof of Theorem \ref{theomain2}.
Our strategy is to apply the result from Theorem \ref{theo-Nie} to the setting of standard quadratic programs. Let us recall the following problem, already introduced in Section \ref{secsketch}:
\begin{align}\label{sqp}\tag{SQP$_{M}$}
\min\{x^TMx: x\in \Delta_n\}
\end{align}
and the corresponding Lasserre hierarchy introduced in relation (\ref{las-sqp}).
Note the optimal value of (\ref{sqp}) is zero as  $M\in \bCOP_n$. 

Now we will apply Theorem \ref{theo-Nie} to problem (\ref{sqp}).  The set $\MK=\Delta_n$ indeed satisfies the Archimedean condition (this is well-known and easy to check; see, e.g.,  \cite{LV2021a}). By \cite[Theorem~3.1]{Marshall}, the feasible region of the Lasserre hierarchy (\ref{las-sqp}) associated to problem (\ref{sqp}) is a closed set. Hence, the `$\sup$' in program (\ref{las-sqp}) can be changed to a '$\max$'. As a consequence, for a matrix  $M\in \bCOP_n$, having finite convergence of the Lasserre hierachy (\ref{las-sqp}) associated to problem (\ref{sqp})  is equivalent to having  $M\in \bigcup_{r\geq 0}\LAS_{\Delta_n}^{(r)}$. So we obtain the following corollary. 

\begin{coro}\label{finite-LAS}
Let $M\in \bCOP_n$. If the optimality conditions (CQC), (SCC) and (SOSC) hold at every global minimizer of problem (\ref{sqp}) then $M\in \bigcup_{r\geq 0} \LAS_{\Delta_n}^{(r)}$.
\end{coro}

As mentioned earlier, our objective is to give sufficient conditions on $M$ that permit to claim  $DMD\in \bigcup_{r\geq 0} \LAS_{\Delta_n}^{(r)}$ for all $D\in \MD_{++}^n$. For this we will apply  Corollary \ref{finite-LAS}. A key fact we will show is that, if the optimality conditions hold at every minimizer of problem (\ref{sqp}) for $M$, then the same holds for $DMD$ for any $D\in \MD^n_{++}$.
Given $D\in \MD^n_{++}$, let us consider the standard quadratic program associated to $DMD$:
\begin{align}\label{Dsqp}\tag{SQP$_{DMD}$}
\min\{x^TDMDx: x\in \Delta_n\}.
\end{align}
Observe that the optimal value of program (\ref{Dsqp}) is zero. Indeed, if $u\in \Delta_n$ is a minimizer of problem (\ref{sqp}), then $\frac{D^{-1}u}{\|{D^{-1}u\|_1}}\in \Delta_n$ is a minimizer of problem (\ref{Dsqp}). Conversely, if $v\in \Delta_n$ is a minimizer of (\ref{Dsqp}), then ${Dv\over \|Dv\|_1}$ is a minimizer of (\ref{sqp}). Hence,  the minimizers of both problems are in one-to-one correspondence, and thus  problem (\ref{sqp}) has finitely many minimizers if and only if problem (\ref{Dsqp}) has finitely many minimizers. 

\medskip
Now we analyze the optimality conditions (CQC), (SCC) and (SOSC) for problems (\ref{sqp}) and (\ref{Dsqp}). Observe that the constraint qualification condition (CQC) is satisfied at every minimizer. Indeed, if $u\in \Delta_n$, then $J(u)=\{i\in [n]: x_i=0\}=[n]\setminus \supp(u)$, and   the vectors $e$, $e_i$ (for $i\in J(u)$) are linearly independent. 

 Let us recall a result from \cite{Diananda} about the support of optimal solutions for problem (\ref{sqp}), which we will  use for the analysis of the conditions (SCC) and (SOSC). We give the short proof for clarity.

\begin{lemma} \cite[Lemma 7 (i)]{Diananda} \label{support-psd}
Let $M\in \COP_n$   and let $x\in \mathbb{R}_+^n$ be such that $x^TMx=0$. Let $S=\supp(x)$ be the support of $x$. Then $M[S]$, the principal submatrix of $M$ indexed by $S$, is positive semidefinite.  
\end{lemma}

\begin{proof}
Let $\tilde{x}=x|_S$ be the restriction of $x$ to the coordinates indexed by $S$, so $\tilde{x}^TM[S]\tilde{x}=0$.  Assume by contradiction that $M[S]$ is not positive semidefinite. Then there exists $y\in \mathbb{R}^{S}$ such that $y^TM[S]y<0$ and we can  assume that {$y^TM[S]\tilde{x} \le 0$} (else replace $y$ by $-y$).  Since all entries of $\tilde{x}$ are positive, there exists $\lambda\geq 0$ such that the vector $\lambda \tilde{x}+y$ has all its entries positive. Thus,  $(\lambda\tilde{x}+y)^TM[S](\lambda\tilde{x}+y)= \lambda^2\tilde{x}^TM[S]\tilde{x}+ 2\lambda\tilde{x}^TM[S]y+ y^TM[S]y <0$, contradicting that $M[S]$ is copositive. 
\end{proof}

We now characterize the minimizers for which the strict complementarity condition (SCC) holds. Moreover, we show that, if a minimizer $u$ of problem (\ref{sqp}) satisfies  (SCC), then the corresponding minimizer $\frac{D^{-1}u}{\norm{D^{-1}u}}$ of problem (\ref{Dsqp}) also satisfies (SCC).

\begin{lemma}\label{scc-diag}
Let  $M\in \bCOP_n$, $D\in \MD_{++}^{n}$, and let $u$ be a minimizer of problem (\ref{sqp}). The  strict complementarity condition (SCC) holds at $u$ if and only if  $\supp(Mu)=[n]\setminus \supp(u)$ or, equivalently,  $(Mu)_i>0$ 
 for all $i\in [n]\setminus \supp(u)$.

As a consequence, (SCC)  holds at $u$ (for problem (\ref{sqp})) if and only if (SCC) holds at $\frac{D^{-1}u}{\norm{D^{-1}u}}$ (for problem (\ref{Dsqp}). \end{lemma}

\begin{proof}
Let $S=\supp(u)$. We first prove that $(Mu)_i=0$ for any $i\in S$. 
 Let $\tilde{u}=u|_{S}$ denote the restriction of vector $u$ to the coordinates indexed by $S$. Then, we have  $0=u^TMu=\tilde{u}^TM[S]\tilde{u}$. By Lemma \ref{support-psd}, $M[S]$ is positive semidefinite, and thus  $\tilde{u}\in\text{Ker}(M[S])$. Thus,  $0=(M[S]\tilde{u})_i=(Mu)_i$ for any $i\in S$. This shows 
 $\supp(Mu)\subseteq [n]\setminus S$.
 Hence equality $\supp(Mu)= [n]\setminus S$ holds if and only if   $(Mu)_i=\sum_{j\in \supp(u)} M_{ij}u_j>0$ for all $i\in [n]\setminus \supp(u)$. It suffices now to  show the link to (SCC).
 
 In problem (\ref{sqp}) the strict complementarity condition (SCC) reads:
$$Mu = \lambda e + \sum_{j\in[n]\setminus S}\mu_je_j \quad \text{ with } \mu_j>0 \text{ for } j\in [n]\setminus S.$$
By looking at the coordinate indexed by $i\in S$ we obtain that $0=(Mu)_i=\lambda$. Hence, $(Mu)_j=\mu_j$ for any $j\in [n]\setminus S$.  Therefore (SCC) holds if and only if  $(Mu)_j>0$ for all $j\in [n]\setminus S$.

The last claim of the lemma follows  using the above characterization, combined with  the correspondence between the minimizers $u$ of (\ref{sqp}) and $D^{-1}u$ (up to scaling) of (\ref{Dsqp}) and the fact that $\supp(Mu)=\supp(DMu)$ and 
$\supp(D^{-1}u)=\supp(u)$ (as $D$ is positive diagonal).
\end{proof}

As observed, e.g., in \cite{Nie}, if the sufficient optimality conditions (CQC), (SCC), (SOSC) hold at every global minimizer, then the number of minimizers must be finite. We now show a useful fact:  if a standard quadratic program has  
 finitely many minimizers, then (SOSC) holds at all of them. 
 
\begin{lemma}\label{finite-sosc}
Let $M\in \bCOP_n$, so that problem (\ref{sqp}) has optimal value zero. If (\ref{sqp})   has finitely many minimizers, then (SOSC) holds  at every global minimizer.
\end{lemma}

\begin{proof}
Assume $M\in \bCOP_n$ and (\ref{sqp}) has finitely many minimizers.
We first prove that, given $S\subseteq [n]$,  problem (\ref{sqp}) has at most one optimal solution with support $S$. For this, assume by contradiction that $u\ne v\in \Delta_n$ are solutions of $x^TMx=0$ with support $S$. By Lemma \ref{support-psd} the matrix $M[S]$ is positive semidefinite. Let $\tilde{u}$ and $\tilde{v}$ be the restrictions of the vectors $u$ and $v$ to the entries indexed by $S$. Hence, $\tilde{u}^TM[S]\tilde{u}=\tilde{v}^TM[S]\tilde{v}=0$, and thus $M[S]\tilde{u}=M[S]\tilde{v}=0$. This implies that every convex combination of $\tilde{u}, \tilde{v}$ belongs to the kernel of $M[S]$, so that the form $x^TM[S]x$ has infinitely many zeroes on $\Delta_{|S|}$. Hence,  $x^TMx$ has infinitely many zeroes on $\Delta_{n}$, contradicting the assumption. 

Let $u$ be a minimizer of problem (\ref{sqp}) with support $S$ and consider as above  its restriction $\tilde u\in \oR^{|S|}$.
 Observe that the \textit{second order sufficiency condition} (SOSC) for problem (\ref{sqp}) at $u$ reads 
$$ v^TMv > 0  \text{ for all } v\in \mathbb{R}^n\setminus \{0\} \text{ such that } \sum_{i=1}^{n}v_i=0 \text{ and } v_j=0\quad \forall j\in [n]\setminus S,$$
$$\text{or,  equivalently, } \quad a^TM[S]a> 0  \text{ for all } a\in \mathbb{R}^{|S|} \setminus {\{0\}} \text{ such that } \sum_{i\in S}a_i=0.$$
Assume that $a^TM[S]a=0$, we show $a=0$. Since $M[S]\succeq 0$ we have that $M[S]a=0$, so that $M[S](\lambda \tilde{u} +a)=0$ for all $\lambda\in \mathbb{R}$. Pick $\lambda>0$ large enough so that all entries of $\lambda \tilde{u}+a$ are positive. Then $\lambda \tilde{u} + a$ should be a multiple of $\tilde{u}$ because $u$ is the only minimizer over the simplex with support $S$. Combining with  the fact that $e^Ta=0$ this implies $a=0$.
\end{proof}

As previously observed, the minimizers of problems (\ref{sqp}) and (\ref{Dsqp}) are in one-to-one correspondence. Thus, as a consequence of Lemma \ref{finite-sosc}, (SOSC) holds at every globlal minimizer of (\ref{sqp}) if and only if it holds at every global minimizer of problem (\ref{Dsqp}). Moreover, we have shown in Lemma \ref{scc-diag} that (SCC) holds for all minimizers of problem (\ref{Dsqp}) if and only if it holds for all minimizers of (\ref{sqp}). Therefore, using  Corollary \ref{finite-LAS}, we obtain the following result.

\begin{theorem}\label{opt-DMD}
Let $M\in \bCOP_n$ and assume problem (\ref{sqp}) has finitely many mininizers. Assume moreover that, for every minimizer $u$ of problem (\ref{sqp}), we have $(Mu)_i>0$ for all $i\in [n]\setminus \supp(u)$.
Then $DMD\in \bigcup_{r\geq0}\LAS_n^{(r)}$ for all $D\in \MD_{++}^n$.
\end{theorem}


\subsection{Proof of Theorem \ref{theomain2}}

Now we can prove the result of Theorem \ref{theomain2}; that is, we show that $DT(\psi)D\in\bigcup_{r\geq0}\LAS_{\Delta_n}^{(r)}$ for all $D\in \MD_{++}^{n}$ and $\psi\in \Psi$. 
We show this result as a direct application of Theorem \ref{opt-DMD}. It thus remains to check that the two assumptions in Theorem \ref{opt-DMD} hold.
First, by combining two results from \cite{Hildebrand}, the description of the (finitely many) minimizers of problem (\ref{sqp}) for $M=T(\psi)$ ($\psi\in \Psi$) can be found.

\begin{lemma}\label{minimizers}
The minimizers of problem (\ref{sqp}) associated to  the matrix $M=T(\psi)$ (with  $\psi\in \Psi$) are the vectors $v_i=\frac{u_i}{\|u_i\|_1}$ for $i\in [5]$, where the $u_i$'s are defined by
\begin{equation*}
\resizebox{0.9\hsize}{!}{$ u_1=\begin{pmatrix}  \sin \psi_5 \\
\sin(\psi_4+\psi_5)\\
\sin \psi_4\\
0\\
0
\end{pmatrix}, \
u_2= 
\begin{pmatrix}
\sin(\psi_3+\psi_4)\\
\sin \psi_3\\
0\\
0\\
\sin \psi_4
\end{pmatrix}, \ 
u_3= \begin{pmatrix}
0\\
\sin \psi_1\\
\sin(\psi_1+\psi_5)\\
\sin\psi_5\\
0
\end{pmatrix},\
u_4=\begin{pmatrix}
0\\
0\\
\sin\psi_2\\
\sin(\psi_1+\psi_2)\\
\sin\psi_1
\end{pmatrix},\
u_5=\begin{pmatrix}
\sin\psi_2\\
0\\
0\\
\sin\psi_3\\
\sin(\psi_2+\psi_3)
\end{pmatrix}$.}
\end{equation*}
\end{lemma}

\begin{proof}
By \cite[Theorem 2.5]{Hildebrand}) it follows that there are exactly five minimizers and that they are  supported, respectively, by  the sets  $\{1,2,3\}, \{1,2,5\}, \{2,3,4\}, \{3,4,5\}$ and $\{1,4,5\}$. Next, using  \cite[Lemma 3.2]{Hildebrand}), we obtain that the minimizers take the desired form.
\end{proof}

We finally check that the second assumption of Theorem \ref{opt-DMD} holds for the matrices  $M=T(\psi)$ ($\psi\in \Psi$).

\begin{lemma}
Let $\psi\in \Psi$ and let $v$ be a minimizer of problem (\ref{sqp}) where $M=T(\psi)$. Then, we have  $(Mv)_i>0$ 
 for all $i\in[5]\setminus \supp(v)$.
\end{lemma}
\begin{proof}
By symmetry, it is enough to check this condition for one of the minimizers, say $v_1$ (as given in Lemma \ref{minimizers}). Since multiplying by a positive constant does not affect the sign we verify the condition for the vector $u_1$. For convenience, we set $u=u_1$. As $\supp(u)=\{1,2,3\}$  the condition we want to check reads as follows
$$\sum_{i=1}^3T(\psi)_{i4}u_i>0\quad  \text{ and }\quad \sum_{i=1}^3T(\psi)_{i5}u_i>0.$$
Again, it suffices to check just the first inequality since the second one is    analogous (up to index permutation). We will now check  that the first expression is positive. Indeed we have
\begin{align*}
&\sum_{i=1}^3T(\psi)_{i4}u_i\\
= &\cos(\psi_2+\psi_3)\sin\psi_5 + \cos(\psi_5+\psi_1)\sin(\psi_4+\psi_5) -\cos\psi_1\sin\psi_4 \\
=& \cos(\psi_2+\psi_3)\sin\psi_5 + (\cos\psi_5\cos\psi_1 -\sin\psi_5\sin\psi_1)(\sin\psi_4\cos\psi_5+\cos\psi_4\sin\psi_5)-\cos\psi_1\sin\psi_4 \\
=& \cos(\psi_2+\psi_3)\sin\psi_5 + (\cos^2\psi_5-1)\cos\psi_1\sin\psi_4+ \cos\psi_5\cos\psi_1 \cos\psi_4\sin\psi_5  \\
 & \hspace{7cm} -\sin\psi_5\sin\psi_1\sin\psi_4\cos\psi_5 -\sin^2\psi_5\sin\psi_1\cos\psi_4 \\
 =& \cos(\psi_2+\psi_3)  \sin(\psi_5) - \sin^2\psi_5\sin(\psi_1+\psi_4) + \sin\psi_5\cos\psi_5\cos(\psi_1+\psi_4) \\
 =&  \cos(\psi_2+\psi_3)\sin\psi_5 + \sin\psi_5 \cos(\psi_1+\psi_4+\psi_5) \\
 =& \sin\psi_5(\cos(\psi_2+\psi_3) + \cos(\psi_1+\psi_4+\psi_5)).
  \end{align*}
  We finish the proof by showing that both factors in the last expression are positive for $\psi\in \Psi$. By the definition of $\Psi$,  $\sum_{i=1}^5 \psi_i<\pi$ and $\psi_i>0$ for $i\in [5]$, so that $ \psi_5\in (0,\pi)$ and thus $\sin\psi_5>0$. Now we use that cosine is a monotone decreasing function in the interval $(0,\pi)$. Observe that $\psi_2+\psi_3$ and  $\pi - (\psi_1+\psi_4+\psi_5)$ belong to $(0,\pi)$ and $\psi_2+\psi_3<\pi - (\psi_1+\psi_4+\psi_5)$. Thus, $\cos(\psi_2+\psi_3)> \cos(\pi-(\psi_1+\psi_4+\psi_5))=-\cos(\psi_1+\psi_4+\psi_5)$,  completing the proof.
  \end{proof}

\section{Concluding remarks}\label{secfinal}

In this paper we   investigate whether  the cones $\MK_n^{(r)}$ provide a complete approximation hierarchy for  the copositive cone $\COP_n$, i.e., whether their union covers the full cone $\COP_n$. As mentioned earlier, the answer is positive for $n\le 4$ (then $\MK_n^{(0)}=\COP_n$  \cite{Diananda}) and negative for $n\ge 6$  \cite{LV2021b}. For the case $n=5$, it was shown in $\cite{DDGH}$ that $\COP_5\neq \MK_5^{(r)}$ for any $r\in \mathbb{N}$. Whether the cones $\MK^{(r)}_5$ provide a complete hierarchy for $\COP_5$, i.e., whether equality 
$\COP_5=\bigcup_{r\geq0}\MK_5^{(r)}$ holds, remains open and is the main topic of this paper. 
As our main result we show that equality holds if and only if every positive diagonal scaling of the Horn matrix $H$ belongs to $\bigcup_{r\geq0}\MK_5^{(r)}$. Our proof technique relies on considering an alternative approximation hierarchy of $\COP_n$, provided by the Lasserre-type cones $\LAS_{\Delta_n}^{(r)}\subseteq \MK_n^{(r)}$. Namely, we show  that all the extreme matrices of $\COP_5$, that do not belong to $\MK^{(0)}_5$ and are not a positive diagonal scaling of the Horn matrix, do indeed  belong to $\bigcup_{r\geq 0}\LAS_{\Delta_n}^{(r)}$. 


\subsubsection*{On the impact of diagonal scaling}
Diagonal scaling plays a crucial role in the analysis of the copositive cone. It is clear that any positive diagonal scaling of a copositive matrix remains copositive. Moreover, a positive diagonal scaling of an extreme copositive matrix remains extreme. However, this operation is not well-behaved with respect to the cones $\MK_n^{(r)}$ when $r\ge 1$ and $n\ge 5$. Dickinson et al. 
\cite{DDGH}  show that, for any matrix $M\in \COP_n\setminus \MK_n^{(0)}$ and any $r\in \mathbb{N}$, there exists a positive diagonal scaling of $M$ that does not belong to $\MK_n^{(r)}$. On the other hand, they also show that every $5\times 5$ copositive matrix with an all-ones diagonal belongs to $\MK_5^{(1)}$. Thus, a method for checking whether a $5\times 5$ matrix belongs to $\COP_5$ is to scale it to obtain a matrix with binary diagonal entries and check whether this new matrix belongs to $\MK_5^{(1)}$. In contrast, as shown in \cite{LV2021b}, for any $n\geq 7$ there exist matrices $M\in \COP_7\setminus \bigcup_{r\ge 0}\MK_n^{(r)}$ with an all-ones diagonal.  The case $n=6$ remains open.
\begin{question}
Let $M\in \COP_6$ with an all-ones diagonal. Does it hold that $M\in \bigcup_{r\geq 0}\MK_6^{(r)}$?
\end{question}

\subsubsection*{Zeros of the form $x^TMx$}
 As shown earlier, for a matrix $M\in \COP_n$, the number of zeros of the form $x^TMx$ in the simplex plays a fundamental role for checking  membership of $M$ in the cones $\LAS_n^{(r)}$. If $M$ is strictly copositive (i.e., $x^TMx$ has no zeroes in $\Delta_n$), then $M\in \bigcup_{r\geq 0}\LAS_{\Delta_n}^{(r)}\subseteq \bigcup_{r\geq 0}\MK_n^{(r)}$. If $M$ has finitely many zeros in $\Delta_n$, then, as was shown in Section \ref{opt-con}, one possible strategy for showing membership in $\bigcup_{r\geq0}\LAS_n^{(r)}$ is checking the classical optimality conditions over  $\Delta_n$ at every zero of $x^TMx$. This was, in fact, our strategy for showing that the matrices $T(\psi)$ for $\psi\in \Psi$ belong to some cone $\LAS_n^{(r)}$ and thus to some cone $\MK_n^{(r)}$.  Finally, if $x^TMx$ has infinitely many zeros in $\Delta_n$, then the classical optimality conditions cannot hold and thus the strategy from Section \ref{opt-con} does not work. One example that shows how the number of minimizers causes issues is the Horn matrix $H$. While $H$ belongs to $\MK_5^{(1)}$, it does not belong to $\LAS_{\Delta_5}^{(r)}$ for any $r\in \mathbb{N}$. To show that, we exploit the structure of the (infinitely many) zeros of the form $x^THx$ on $\Delta_n$. Hence, another strategy will be needed for settling the question whether any diagonal scaling of $H$ belongs to some cone $\MK^{(r)}_5$.

\subsubsection*{Copositive matrices from graphs}
The cones $\MK_n^{(r)}$ are used by de Klerk and Pasechnik \cite{dKP2002} for defining a hierarchy of upper bounds $\vartheta^{(r)}(G)$ for the stability number $\alpha(G)$ of a graph $G$. They conjectured that these bounds converge  to $\alpha(G)$ in $\alpha(G)$ steps. Define the {\em graph matrix}  
$M_G=\alpha(G)(A_G+I)-J$, where $A_G, I$ and $J$ are, respectively, the adjacency, identity and all-ones matrices. 
As an application of a result of Motzkin-Straus \cite{motzkin},   $M_G$ is a copositive matrix and the conjecture in  \cite{dKP2002} boils down to claiming that $M_G$ belongs to the cone $\MK_n^{(\alpha(G)-1)}$. 
In fact, it is not even known whether $M_G$ belongs to {\em some} cone $\MK_n^{(r)}$. So the following two conjectures remain open.
\begin{conjecture}[\cite{dKP2002}]\label{conjecture-dKP}
For any graph $G$, we have $M_G\in \MK_n^{(\alpha(G)-1)}$.
\end{conjecture}
\begin{conjecture}[\cite{LV2021a}]\label{conjecture-LV}
For any graph $G$, we have  $M_G\in \bigcup_{r\geq 0}\MK_n^{(r)}$.
\end{conjecture}
 The subcones $\MQ_n^{(r)}$ and $\LAS_{\Delta_n}^{(r)}$ of $\MK^{(r)}_n$, defined in (\ref{eqLASDa}) and in (\ref{eqQra}),  have been used to  partially resolve these two conjectures. On the one hand, Laurent and Gvozdenovi\'c  \cite{GL2007} established  Conjecture \ref{conjecture-dKP} for graphs with $\alpha(G)\leq 8$, by  showing $M_G\in \MQ_n^{(\alpha(G)-1)}$ for these graphs 
  (see also \cite{PVZ2007} for the case $\alpha(G)\leq 6$).  On the other hand,  we showed in \cite{LV2021a} that $M_G$ belongs to some cone $\LAS_{\Delta_n }^{(r)}$ whenever $G$ has no critical edges, i.e., when $\alpha(G\setminus e) = \alpha(G)$ for any edge $e$. Observe that the Horn matrix coincides with the graph matrix $M_{C_5}$ when $G=C_5$ is   the 5-cycle. 
  In fact, by Lemma \ref{lemHornLASD}, $M_{C_5}\notin \bigcup_{r\geq0}\LAS^{(r)}$, but $M_{C_5}\in  \MK^{(1)}_5=\MQ^{(1)}_5$. Notice also that $C_5$ is critical (i.e., all  its edges are critical). As observed in \cite{LV2021a} it in fact suffices to show Conjectures \ref{conjecture-dKP} and \ref{conjecture-LV} for the class of critical graphs.

The class of graph matrices $M_G$ has been recently further investigated in \cite{DZ}, where they are used, in particular,  to construct large classes of extreme matrices of $\COP_n$.

\end{document}